\documentclass{amsart}

\usepackage{amsmath,amsthm,amsopn,amstext,amscd,amsfonts,amssymb,mathrsfs,mathtools}
\usepackage{dsfont} % para escribir correctamente los cjtos R,Z,C...
\usepackage{comment}
\usepackage[active]{srcltx}
\usepackage{graphicx, mathdots}
\DeclareMathOperator{\rank}{rank}
  \DeclareMathOperator{\diag}{diag}

\newtheorem{definition}{\sc Definition}[section]

\newtheorem{theorem}{\sc Theorem}[section]

\newtheorem{lemma}{\sc Lemma }[section]
\newtheorem{eje}{\sc Example }[section]
\newtheorem{coro}{\sc Corollary}[section]

\usepackage[active]{srcltx}

\newcommand{\de}{\mathbb{D}}

\newcommand{\co}{\mathbb{C}}

\newcommand{\adj}{{\rm Adj\,}}
\newcommand{\dps}{\displaystyle}

\makeatletter\renewcommand\theenumi{\@roman\c@enumi}\makeatother

\usepackage{graphicx, epsfig, subfig}
\usepackage{lscape}
\usepackage{rotating}
\usepackage{caption}
\usepackage{multicol}
\usepackage{colortbl}

\newmuskip\pFqskip
\mathchardef\pFcomma=\mathcode`, % keep a copy of the comma

\begin{document}
\title[Refined interlacing properties]{Refined interlacing properties for zeros of paraorthogonal polynomials on the unit circle}
\author{K. Castillo}
\author{J. Petronilho}

\address{CMUC, Department of Mathematics, University of Coimbra, 3001-501 Coimbra, Portugal}

\email[K. Castillo]{kenier@mat.uc.pt}
\email[J. Petronilho]{josep@mat.uc.pt}

\subjclass[2010]{15A42}
\date{\today}
\keywords{Paraorthogonal polynomials on the unit circle, zeros, unitary matrices, eigenvalues, interlacing, rank one perturbations.}
\begin{abstract}
 The purpose of this note is to extend in a simple and unified way the known results on interlacing of zeros of paraorthogonal polynomials on the unit circle. These polynomials can be regarded as the characteristic polynomials of any matrix similar to an unitary upper Hessenberg matrix with positive subdiagonal elements. 
\end{abstract}
\maketitle

\section{Introduction and main result}
\label{sec:intro}
The study of zeros of {\em orthogonal polynomials on the real line} (OPRL) can be regarded as an eigenvalue problem for Jacobi matrices\footnote{A symmetric tridiagonal matrix whose next-to-diagonal elements are positive (cf. \cite[p. $36$]{HJ}).}.  
This allows us to go back to one of the most important single books in the nineteenth century, {\em Cours d'analyse de l'\'Ecole royale polytechnique} (1821) by Cauchy to deduce, at least in the weak sense, the zero interlacing property of consecutive OPRL from the simplest form of the nowadays called Cauchy interlacing theorem. The search of more refined eigenvalue interlacing properties of Jacobi matrices was probably initiated by Cauchy himself in his work {\em Sur l' \'Equation \`a l' Aide de Laquelle on D\'etermine les Inegalit\'ees S\'eculaires des Mouvements des Plan\`etes} (1829) and later continued by several authors, including in the second half of the last century Wilkinson \cite{W88}, Kahan \cite{Ka66}, Golub \cite{G72}, Hill and Parlett \cite{HP92}, and Bar-On \cite{B96}. In the same spirit, this work recovers one of the earliest approaches used to study zeros of {\em paraorthogonal polynomials on the unit circle} (POPUC), which is based on an eigenvalue problem for certain unitary matrices which bear many similarities with Jacobi matrices (cf. \cite{K85, AGR86, G86, AGR88, GR90, BE91, DG91b, W93, B93, BH95, CMV03, S05I, S05II, KN07, S07, S07b, S11}).

Without wishing to delve into a historical discussion\footnote{The weakened orthogonality condition that POPUC satisfy appeared in \cite[Equation $4.10$]{DG88} as far as we can tell. While it is true that in Geronimus' 1944 paper \cite[Theorem IV]{G44} such polynomials were presented.}, as far as we know, the POPUC\footnote{In \cite{DG88}, Delsarte and Genin called to these polynomials (symmetric) {\em predictor polynomials} and its weakened orthogonality property {\em quasiorthogonality}. In \cite{DG90}, they refer to these polynomials as {\em quasiorthogonal polynomials on the unit circle}. This denomination could be also supported by the fact that in 1946 Geronimus regarding to these polynomials wrote that they {\em ``...play the same role here as the quasi-orthogonal polynomials of M. Riesz in the Hamburger problem."} (cf. \cite[Remark I]{G46}). The denomination POPUC was coined in \cite{JNT89}.}   were introduced  (in a somewhat hidden form) and successfully developed in a serie of papers by Delsarte and Genin at the end of the 1980's \cite{DG88, DG91a, DG91b}, when they were working in signal processing. In \cite{DG91b}, the authors focuses on  the problem of computing the zeros of POPUC regarded as an eigenvalue problem for an unitary upper Hessenberg matrix with positive subdiagonal elements. Elegant and recent proofs of most interlacing properties of zeros of POPUC shared with OPRL are due to Simon \cite{S07} (cf. \cite[Theorem $2.14.4$]{S11}) where the theory of rank one perturbations plays a central role. However, before such work (and references therein)  the zeros of POPUC were studied by the Linear Algebra community based on  ideas close to those of Simon but supported on more elementary facts. Further analysis of these ideas will allow us to easily extend the known results. Indeed, our main purpose is to prove and improve, in connection with the works of Delsarte and Genin on the subject, the known zero interlacing properties of POPUC, based on the development of the ideas discussed by Arbenz and Golub  in \cite[Section $6$]{AGo88}\footnote{Such ideas were pioneering employed in the present context by Bohnhorst in her Ph.D. thesis \cite{B93} defended in 1993 at the Bielefeld University under the supervision of Elsner.}.

Here and below, we mainly follow the notation of \cite{S05I, S05II, S11}. Denote by $\de$ the open unit disk and by $\mathbb{S}^1$ its boundary, i.e.,$$\de:=\{z \in \co : |z|<1\}\,\, ,\quad \mathbb{S}^1:=\{z \in \co : |z|=1\}\, .$$
Let $(\alpha_0, \dots, \alpha_{n-1}, b_n)$ with $\alpha_j \in \de$ ($j=0,1,\dots,n-1$) and $b_n \in \mathbb{S}^1$. Set
\begin{align*}%\label{alpha}
\Theta_j:=\Theta(\alpha_j), \quad \Theta(\alpha):=\begin{pmatrix}
\overline{\alpha} & \rho \\
\rho & -\alpha
\end{pmatrix}\,, \quad \rho:=\left(1-|\alpha|^2\right)^{1/2}\, .
\end{align*}
Define the $(n+1)$-by-$(n+1)$ matrix 
\begin{align}\label{lm}
\mathcal{C}:=\mathcal{L}\,\mathcal{M}\, ,
\end{align}
where $\mathcal{L}$ and $\mathcal{M}$ are given explicitly by
\begin{align*}
\mathcal{L}&:=\left\{ 
\begin{array}{l l}
\Theta_0 \oplus \Theta_2 \oplus \cdots \oplus \Theta_{n-2}  \oplus \overline{b}_{n}& \ \ \ \ \ \ \ \mbox{if $n$ is even}\\[0.5em]
 \Theta_0 \oplus \Theta_2 \oplus \cdots \oplus \Theta_{n-1}  & \ \ \ \ \ \ \ \mbox{if $n$ is odd}\, ,\\ \end{array} \right.\\
\mathcal{M}&:=\left\{
\begin{array}{l l}
1 \, \oplus \, \Theta_1 \oplus \Theta_3 \oplus \cdots \oplus \Theta_{n-1}    & \ \mbox{if $n$ is even }\\[0.5em]
1 \, \oplus \, \Theta_1 \oplus \Theta_3 \oplus \cdots \oplus \Theta_{n-2} \oplus \overline{b}_{n}  &\ \mbox{if $n$ is odd}.
 \end{array} \right.
\end{align*}
Any unitary $(n+1)$-by-$(n+1)$ upper Hessenberg matrix with positive subdiagonal elements is uniquely parameterized by $2n+1$ real numbers that compose the parameters of the array $(\alpha_0, \dots, \alpha_{n-1}, b_n)$ \cite{G82o} (cf. \cite{G82} and \cite[Proposition $1$]{AG92}). The resulting matrix after this process is referred as the Schur parametric form of the original matrix. The factorization \eqref{lm}, which is unitarily similar to the Schur parametric form of an upper Hessenberg matrix with positive subdiagonal elements, was presented by Bunse-Gerstner and Elsner \cite{BE91} (cf. \cite[Section $12.2.10$]{GV83} and \cite[Definition $3.3$ and Lemma $3.4$]{B93}).  The explicit unitary pentadiagonal or double-staircase form of $\mathcal{C}$ (referred as {\em Doppel-Treppen-Matrix} in the original German source) was studied extensively by Bohnhorst \cite{B93}, see Figure \ref{c} for an $8$-by-$8$ example (cf. \cite[Equation $3.9$]{B93} and \cite[Figure $1.1$]{KN07}). The matrix $\mathcal{C}$ becomes a very popular object in the Mathematical Physics and Orthogonal Polynomials communities after the work \cite{CMV03}, specially after Simon's monographs \cite{S05I, S05II}  where it was called (improper) CMV matrix (cf. \cite{S07b, S11}).

\begin{figure}[h]
$$
\begin{pmatrix}
\overline{\alpha}_0 & \rho_0 \overline{\alpha}_1 & \rho_0 \rho_1 &  &  &  &  &  \\
\rho_0 & -\alpha_0 \overline{\alpha}_1 &  -\alpha_0 \rho_1 &  &  &  &  &  \\
& \rho_1 \overline{\alpha}_2 & -\alpha_1 \overline{\alpha}_2 &  \rho_2 \overline{\alpha}_3 & \rho_2 \rho_3 &  &  &  \\
& \rho_1 \rho_2 & -\alpha_1 \rho_2 &  -\alpha_2 \overline{\alpha}_3 & -\alpha_2 \rho_3 &  &  &  \\
 &  &  & \rho_3 \overline{\alpha}_4 & -\alpha_3 \overline{\alpha}_4 &  \rho_4 \overline{\alpha}_5 & \rho_4 \rho_5 &  \\
 &  &  & \rho_3 \rho_4 & -\alpha_3 \rho_4 &  -\alpha_4 \overline{\alpha}_5 & -\alpha_4 \rho_5 &  \\
 &  &  &  &  & \rho_5 \overline{\alpha}_6 &  -\alpha_5 \overline{\alpha}_6 & \rho_6 \overline{b}_7 \\
 &  &  &  &  & \rho_5 \rho_6 &  -\alpha_5 \rho_6 & -\alpha_6 \overline{b}_7
\end{pmatrix}
$$
\caption{The matrix $\mathcal{C}$ for $n=7$.}
\label{c}
\end{figure}
In order to make the notation more transparent, we write $\mathcal{C}(\alpha_0,\dots,$ $\alpha_{n-1},b_n)$ instead of $\mathcal{C}$. We choose the representation \eqref{lm} instead of their unitary similar upper Hessenberg matrix for a technical reason related to the manner in which  Lemma \ref{AG} below is presented. In the next definition and subsequently, $\mathcal{I}$ denotes the identity matrix, whose order is made explicit or may be inferred from the context.
 \begin{definition}[cf. {\cite[Proposition $3.2$]{S07}}]\label{POPUC}
Let $\mathcal{C}(\alpha_0,\dots,\alpha_{n-1},b_n)$ be the matrix given by \eqref{lm}, where $\alpha_j \in \de$ ($j=0,1,\dots,n-1$) and $b_n \in \mathbb{S}^1$. The (monic) polynomial $P_{n+1}$ defined by
 \begin{align*}
P_{n+1}(z):=\det \big(z \mathcal{I}-\mathcal{C}(\alpha_0,\dots,\alpha_{n-1},b_n)\big)
 \end{align*}
is the POPUC of degree $n+1$ associated with the array $(\alpha_0,\dots,\alpha_{n-1},b_n)$. 
\end{definition}

It is not difficult to see that the eigenvalues of $\mathcal{C}(\alpha_0,\dots,\alpha_{n-1},b_n)$ are simple.  This fact was observed in $1944$ by Geronimus \cite[Theorem IV]{G44}  (cf.  \cite[Theorem\,III]{G46}, \cite[Theorem $9.1$]{G54} and \cite[Theorem $7.2.2$]{A64}) using the connection between POPUC and OPUC. Note that if $b_n$ were in $\de$, then the corresponding characteristic polynomial would be an OPUC and their zeros would be in $\de$.  A remarkable property of the eigenvectors of $\mathcal{C}(\alpha_0,\dots,\alpha_{n-1},b_n)$ is the fact that all their components are nonzero  (cf. \cite[Chapter $4$]{S05I} and references therein). This property is clearly valid also for the corresponding unitarily similar Hessenberg matrix. %For more details we refer the reader to \cite{W93}, \cite[Theorem $3.1$]{CMV03}, \cite[Chapter $4$]{S05I}, \cite{S07b}, and the reference therein.

\begin{definition}
Two finite subsets $\{\zeta_1, \zeta_2, \dots, \zeta_{m}\}$ and $\{\xi_1, \xi_2, \dots, \xi_{n}\}$ $(1\leq m\leq n)$ of $\mathbb{S}^1$ interlace (resp. strictly interlace)  whenever there exist $n-m$ points $\zeta_{m+1}, $ $\zeta_{m+2}, \dots, \zeta_{n} \in \mathbb{S}^1$ such that any closed arc  (resp. open arc) on $\mathbb{S}^1$ connecting two distinct elements of $\{\zeta_1, \zeta_2, \dots, \zeta_{n}\}$ contains at last one element of $\{\xi_1, \xi_2, \dots, \xi_{n}\}$, and vice versa.
\end{definition}

We can now formulate our main result.

\begin{theorem}\label{1}
Let $\mathcal{C}(\alpha_0,\dots,\alpha_{n-1},b_n)$ be a matrix given by \eqref{lm}, where $\alpha_j \in \de$ ($j=0,1,\dots,n-1$) and $b_n \in \mathbb{S}^1$. The following sentences hold:
\begin{enumerate}
\item [{\rm (i)}] Let $\beta \in \mathbb{S}^1 \setminus \{1\}$ and define $\mathcal{C}_m^{\beta}:=\mathcal{C}(\alpha_0,\dots, \alpha_{m-1}, \beta \alpha_{m},\dots,$ $ \beta \alpha_{n-1}, \beta b_n)$ ($0 \leq m <n$) and $\mathcal{C}_n^\beta:=\mathcal{C}(\alpha_0,\dots, \alpha_{n-1},\beta b_n)$. Then the eigenvalues of $\mathcal{C}(\alpha_0,\dots,\alpha_{n-1},b_n)$ and $\mathcal{C}_m^\beta$ strictly interlace on $\mathbb{S}^1$ for each $0\leq m \leq n$.
\item [{\rm (ii)}] For each $0\leq m < n$, let $b_m \in \mathbb{S}^1$, and let $\mathcal{C}(\alpha_0,\dots,\alpha_{n-1},b_n)$ be  partitioned as
\begin{align}\label{C}
\mathcal{C}(\alpha_0,\dots,\alpha_{n-1},b_{n})=\begin{pmatrix}
\mathcal{C}_{11} & \mathcal{C}_{12}\\
\mathcal{C}_{21} & \mathcal{C}_{22}
\end{pmatrix}\, ,
\end{align}
$\mathcal{C}_{11}$ being the $(m+1)$-by-$(m+1)$ leading principal submatrix of $\mathcal{C}(\alpha_0,\dots,$ $\alpha_{n-1},b_n)$. For each $\zeta \in \mathbb{S}^1$, define recursively the numbers\footnote{In \cite{DG91a} (cf. \cite[Equation $2.6$]{DG91b}), Delsarte and Genin have shown that if the $b_{j}(\zeta)$'s (known as {\em pseudo reflection coefficients}) are given by \eqref{def}, then the corresponding POPUC satisfy a three-term recurrence relation (cf. \cite{CCP16}). Bunse-Gerstner and He \cite{BH95} have provided an illuminating discussion of the works of Delsarte and Genin on POPUC in matrix terms.} 
\begin{align}\label{def}
  b_{n}(\zeta):=b_n, \quad
 b_{j}(\zeta):=\frac{\overline{\zeta} \ \alpha_j+ b_{j+1}(\zeta)}{\overline{\alpha}_j b_{j+1}(\zeta) +\overline{\zeta}} \, \quad  (j=n-1,\dots,1, 0) \,.
 \end{align}
 Set\footnote{$\sigma(\mathcal{A})$ denotes the spectrum of $\mathcal{A}$.} 
\begin{align*}
{A}:={C} \, \cap \, \sigma(\mathcal{C}(\alpha_0,\dots,\alpha_{m-1},b_m)), \quad
{B}:=\sigma(\mathcal{N}) \, \backslash \, {A}\, ,
\end{align*}
where ${C}:=\{\zeta \in \mathbb{S}^1: b_m(\zeta)=b_m\}$, $\mathcal{N}:=\mathcal{C}(\alpha_{m+1},\dots,\alpha_{n-1},b_{n})  \, \mathcal{D}$ with $\mathcal{D}:=\diag \, (\gamma_m, \mathcal{I})$,  and
 \begin{align}\label{auxgamma}
 \gamma_{m}:=\frac{\alpha_{m}-b_{m}}{ \overline{\alpha}_{m}b_{m}-1}\, .
\end{align}
Then $\mathcal{C}(\alpha_0,\dots,\alpha_{n-1},b_{n}) $ and $\mathcal{C}(\alpha_0,\dots,\alpha_{m-1},b_{m})$ have at most $\min \{m+1,n-m\}$ common eigenvalues. More precisely, $\mathcal{C}(\alpha_0,\dots,$ $\alpha_{n-1},b_{n}) $ and $\mathcal{C}(\alpha_0,\dots,\alpha_{m-1},b_{m})$ have ${A}$ as the set of common eigenvalues, ${A}$  being also given by the alternative expression
$$
{A}=\sigma(\mathcal{N}) \, \cap \, \sigma(\mathcal{C}(\alpha_0,\dots,\alpha_{m-1},b_m))\, .
$$ 
Furthermore, the elements of the sets $\sigma\big(\mathcal{C}$ $(\alpha_0,\dots,$ $\alpha_{n-1},b_n)\big) \setminus {A}$  and $\sigma\big(\mathcal{C}(\alpha_0,$ $\dots,\alpha_{m-1},$ $b_{m})\big)\, \cup \, {B}$ strictly interlace on $\mathbb{S}^1$.   \end{enumerate}

\end{theorem}
Let $P_{n+1}$ be the POPUC of degree $n+1$ associated to the array $(0,\dots,0,1)$. Since $\mathcal{C}(0,\dots,0,1)$ is a permutation matrix, it follows that $P_{n+1}(z)=z^{n+1}-1$. The sequence $(P_j)_{j\geq 1}$ (all of whose zeros are roots of unity) produce, by geometric intuition, illuminating examples that fall within Theorem \ref{1}.

\begin{eje}\label{permutation}%[Permutation matrices]
 Let $P_3$ and $P_6$ be the POPUC associated to the arrays $(0,0,1)$ and $(0,0,0,0,0,1)$, respectively. In this situation,
  $$ \mathcal{C}(0,0,1)=\begin{pmatrix}
  0 & 0 & 1\\
  1 & 0 & 0\\
  0 & 1 & 0
  \end{pmatrix}, \quad
 \mathcal{C}(0,0,0,0,0,1)=\begin{pmatrix}
 0 & 0 & 1 & 0 & 0 & 0 \\
 1 & 0 & 0 & 0 & 0 & 0\\
 0 & 0  & 0 & 0 & 1 & 0\\
0 & 1 & 0 & 0 & 0 & 0\\
 0 & 0 & 0 & 0 & 0 & 1\\
 0 & 0 & 0 & 1 & 0 & 0\\
 \end{pmatrix}\,,
$$
and, therefore,
\begin{align*}
\sigma( \mathcal{C}(0,0,1))=\left\{1, e^{\pm i 2 \pi/3} \right\}, \quad \sigma( \mathcal{C}(0,0,0,0,0,1))=\left\{\pm1, \dps e^{\pm i 2 \pi/3}, \dps e^{\pm i \pi/3} \right\}\, .
\end{align*}
In the notation of Theorem \ref{1} we have $n=5$, $m=2$, $b_j(\zeta)=\zeta^{5-j}$ $(0\leq j \leq5)$,
${A}={C}=\sigma( \mathcal{C}(0,0,1))$, and ${B}= \emptyset$, where ${A}$ is obtained by using any of the expressions outlined in Theorem \ref{1}. Clearly, $\mathcal{C}(0,0,0,0,0,1)$ and $\mathcal{C}(0,0,1)$ have ${A}$ as the set of common eigenvalues and the elements of the sets $\sigma(\mathcal{C}(0,0,0,0,0,1)) \setminus {A}$  and $\sigma( \mathcal{C}(0,0,1))$ strictly interlace on $\mathbb{S}^1$,  in concordance with sentence (ii) of Theorem \ref{1}.
\end{eje}

Regarding Theorem \ref{1}, as far as we know, sentence (i) for $m=0$ was proved by Ammar, Gragg and Reichel \cite[Proposition $4.2$]{AGR88}, although the particular case $\beta=-1$ is known since Geronimus' work \cite[Theorem IV]{G44} (cf. \cite[Theorem\,III]{G46}). The sentence (i) for $m=n$ was proved by Bohnhorst in \cite[Theorem $3.19$]{B93} (cf. \cite[Theorem $3.5$]{BBF00}). In  \cite[Theorem $3.4$]{S07}, Simon proved a weaker version of sentence (ii) that reads as follows:  {\em Strictly between any pair of eigenvalues of $\mathcal{C}(\alpha_0,\dots,\alpha_{m-1},b_m)$ there is at least one eigenvalue of $\mathcal{C}(\alpha_0,\dots,\alpha_{n-1},b_n)$}.

\begin{coro}\label{caso}
Let $\mathcal{C}(\alpha_0,\dots,\alpha_{n-1},b_n)$ be a matrix given by \eqref{lm}, where $\alpha_j \in \de$ ($j=0,1,\dots,n-1$) and $b_n \in \mathbb{S}^1$. Let $b_{n-1} \in \mathbb{S}^1$ and define $\gamma_{n-1}$ as in \eqref{auxgamma} for $m=n-1$. Then $\mathcal{C}(\alpha_0,\dots,\alpha_{n-1},$ $b_{n}) $ and $\mathcal{C}(\alpha_0,\dots,\alpha_{n-2},b_{n-1})$ have at most one common eigenvalue. More precisely, either $\mathcal{C}(\alpha_0,\dots,\alpha_{n-1},b_n)$ and $\mathcal{C}(\alpha_0,\dots,\alpha_{n-2},b_{n-1})$ have $\overline{b}_n \gamma_{n-1}$ as (only) common eigenvalue and the elements of $\sigma\big(\mathcal{C}(\alpha_0,\dots,\alpha_{n-1},b_n)\big) \setminus \{\overline{b}_n \gamma_{n-1}\}$  and $\sigma\big(\mathcal{C}(\alpha_0,\dots,\alpha_{n-2},$ $b_{n-1})\big)$ strictly interlace on $\mathbb{S}^1$, or else  $\mathcal{C}(\alpha_0,\dots,$ $\alpha_{n-1},b_n)$ and $\mathcal{C}(\alpha_0,\dots,\alpha_{n-2},b_{n-1})$ have no common eigenvalues, and in such case $\overline{b}_n \gamma_{n-1}$ is not an eigenvalue of either, and the elements of the sets $\sigma\big(\mathcal{C}(\alpha_0,$ $\dots,\alpha_{n-1},$ $b_n)\big)$  and $\sigma\big(\mathcal{C}(\alpha_0,$ $\dots, \alpha_{n-2},b_{n-1})\big) \cup \{\overline{b}_n \gamma_{n-1}\}$ strictly interlace on $\mathbb{S}^1$.
\end{coro}
\begin{proof}
Take $m=n-1$ in Theorem \ref{1}. Hence, \eqref{def} and \eqref{auxgamma} yield ${C}=\{\overline{b}_n \gamma_{n-1}\}$ which, in turn, is equal to $\sigma(\mathcal{N})$. Then either ${A}={C}$ and ${B}=\emptyset$ if $\overline{b}_n \gamma_{n-1} \in \sigma\big(\mathcal{C}(\alpha_0,\dots,\alpha_{n-2},b_{n-1})\big)$, or else ${A}=\emptyset$ and ${B}={C}$ otherwise. The result follows immediately from sentence (ii) of Theorem \ref{1}.
\end{proof}

Corollary \ref{caso} was proved by Bohnhorst \cite[p. $48$]{B93} (cf. \cite[p. $819$]{BBF00}) and rediscovered by Simon \cite[Theorem $1.4$]{S07}. It is worth noting that in view of Corollary \ref{caso} and besides the several and well-known practical consequences, POPUC answered the following open-ended question proposed by Tur\'an as far as 1974 \cite[Problem LXVI, p. 60]{T74o}:  {\em ``It is known that the zeros of the $n$th orthogonal polynomial (with respect to a Lebesgue-integral function on an interval) separate the zeros of the $(n+1)$th polynomial. What corresponds to this fact  on the unit circle?"}\footnote{We quote the English translation provided by Sz\"usz \cite[Problem LXVI]{T80}.}.

\section{Proof of Theorem \ref{1}}
\subsection{Some preliminary lemmas}
Theorem \ref{1} will be proved through the following sequence of lemmas. 

\begin{lemma}\label{AG}%[Arbenz {\em \&} Golub; 1988]
Let $\mathcal{U}$ and $\mathcal{S}$ be unitary matrices of the same order and suppose that
$
\rank \ (\mathcal{I}-\mathcal{S})=1\,.
$
Then $\mathcal{U}$ and  $\mathcal{US}$ have interlacing eigenvalues on $\mathbb{S}^1$.  Moreover, assume that $\mathcal{U}\mathcal{S}$ admits a decomposition 
 $
 \mathcal{U}\mathcal{S}=\mathcal{U}_1 \, \oplus \, \mathcal{U}_2\,,
 $
and let $\mathcal{U}$ be  partitioned as
\begin{align*}
\mathcal{U}=\begin{pmatrix}
\mathcal{U}_{11} & \mathcal{U}_{12}\\
\mathcal{U}_{21} & \mathcal{U}_{22}
\end{pmatrix}\, ,
\end{align*}
$\mathcal{U}_{11}$ and $\mathcal{U}_{1}$ being of the same order. Set $U_1:=\sigma(\mathcal{U}_1)$, $U_2:=\sigma(\mathcal{U}_2)$, and  $U:=\sigma(\mathcal{U})$. Assume further that the eigenvalues of $\mathcal{U}_{1}$ and  $\mathcal{U}_{2}$ are simple and $\sigma(\mathcal{U}_{11}) \, \cap \, U_1=\sigma(\mathcal{U}_{22}) \, \cap \, U_2=\emptyset$. Then the elements of the sets $U \setminus \big(U_1 \, \cap \, U_2 \big)$  and $U_1 \, \cup \, \big(U_2 \backslash (U_1 \, \cap \, U_2)\big)$ strictly interlace on $\mathbb{S}^1$.

 \end{lemma}\label{lemma}

\begin{proof}
The first sentence of the lemma is the simplest form of a result due to Arbenz and Golub \cite[Section $6$]{AGo88} (cf. \cite[Theorem $2.9$]{B93} and \cite[Theorem $2.7$]{BBF00})\footnote{It can be deduced directly using \cite[p. $222$]{MMRR14} and \cite[Corollary  $4.3.9$]{HJ}.}.
In order to deduce the second one, we first claim that
\begin{align}\label{claim}
U_1 \, \cap \, U_2=U_1 \, \cap \, U\,=U_2 \, \cap \, U\,.
\end{align}
Indeed,  since $\rank \, (\mathcal{US}-\mathcal{U})=1$, there exist nonzero vectors $u, v \in \co^{n}$ ($n$ being the common order of $\mathcal{U}$ and $\mathcal{S}$) such that
$
\mathcal{US}=\mathcal{U}+u v^T
$. Using the formula for the determinant of a rank one perturbation (cf. \cite[Proposition $3.21$]{S10}), we may write for each $\zeta \in \co$ \footnote{$\chi_{_\mathcal{A}}$ denotes the characteristic polynomial of $\mathcal{A}$.}
\begin{align}\label{deteq}
\chi_{_\mathcal{U}}(\zeta)=\chi_{_\mathcal{US}}(\zeta)+v^T \adj (\zeta \mathcal{I}-\mathcal{US}) u\,.
\end{align}
Let $\mathcal{US}=\mathcal{Z} \Lambda \mathcal{Z}^*$ be the spectral decomposition of $\mathcal{US}$ in which $\Lambda=\diag \,(\lambda_1, \dots,\lambda_n)$ and $\mathcal{Z}=(z_1 \dots z_n)$. Thompson-McEnteggert's formula for the adjugate \cite{TM68} (cf. \cite[Theorem $2.1$]{S79}) gives
\begin{align}\label{th}
\adj (\lambda_j \mathcal{I}-\mathcal{US})=\chi'_{_\mathcal{US}}(\lambda_j) z_j z_j^*\, ,
\end{align}
where the prime denotes the derivative. Combining \eqref{deteq} with \eqref{th} yields\footnote{The eigenvalue interlacing already stated implies $U_1 \, \cap \, U_2 \subseteq U$, and so $U_1 \, \cap \, U_2 \subseteq U_1 \, \cap \, U$ and  $U_1 \, \cap \, U_2 \subseteq U_2 \, \cap \, U$.}
\begin{align}\label{equal}
\chi_{_\mathcal{U}}(\lambda_j)=\left( \chi'_{_{\mathcal{U}_1}}(\lambda_j) \chi_{_{\mathcal{U}_2}}(\lambda_j)+\chi_{_{\mathcal{U}_1}}(\lambda_j) \chi'_{_{\mathcal{U}_2}}(\lambda_j)\right) \, z_j^* u v^T z_j\, .
\end{align} 
We next claim that if $\lambda_j \in (U_1-U_2) \, \cup \, (U_2-U_1)$\footnote{Given a set $E$ and $F,G\subseteq E$, we define $F-G:=F \cap (E \backslash G)$; if $G\subseteq F$, then $F-G=F\backslash G$.}, then $ z_j^* u v^T z_j\not = 0$.  We only prove that $\lambda_j \in U_1-U_2$ implies $v^T z_j \not = 0$. (To prove that  $\lambda_j \in U_1-U_2$ implies $z_j^* u \not = 0$, we proceed similarly, as well as for proving that $\lambda_j \in U_2-U_1$ implies $ z_j^* u v^T z_j\not = 0$.) Indeed, suppose that $\lambda_j \in U_1-U_2$ and $v^T z_j = 0$. Since there is a normalized eigenvector $v_{j}$ of $\mathcal{U}_1$ associated with $\lambda_j$ such that $z_j=(v_{j}^T, 0, \dots, 0)^T$, we deduce  
$$
\mathcal{U}_{11}  \, v_{j}=\lambda_j v_{j}\,,
$$
hence $\lambda_j \in \sigma(\mathcal{U}_{11}) \, \cap \, U_1$, contrary to $ \sigma(\mathcal{U}_{11}) \, \cap \, U_1=\emptyset$. Consequently, \eqref{claim} follows from \eqref{equal}. Finally, it follows from \eqref{claim} that the sets $U \setminus \big(U_1 \, \cap \, U_2 \big)$  and $U_1 \, \cup \, \big(U_2 \backslash (U_1 \, \cap \, U_2)\big)$ have no common elements, thus the second sentence of the lemma follows from the first one.
\end{proof}

\begin{lemma}\label{ideafinal}
Let $\mathcal{U}$ be a unitary matrix and for a fixed $k$ let $\mathcal{S}$ be the diagonal matrix obtained from the identity matrix by replacing the $(k,k)$ entry with a number on $\mathbb{S}^1\setminus \{1\}$. Assume that $\mathcal{U}$ and $\mathcal{S}$ have the same order. Assume further that the eigenvalues of $\mathcal{U}$ are simple and all its eigenvectors have a nonzero component at the position $k$. Then $\mathcal{U}$ and  $\mathcal{U}\mathcal{S}$ have strictly interlacing eigenvalues on $\mathbb{S}^1$.
\end{lemma}
\begin{proof}
Without loss of generality we can assume that $k=1$, and so $\mathcal{S}= \diag \, (\beta,\mathcal{I})$ with $\beta \in \mathbb{S}^1 \setminus \{1\}$. Let $\mathcal{U}=\mathcal{Z} \Lambda \mathcal{Z}^*$ be the spectral decomposition of $\mathcal{U}$ in which $\Lambda=\diag \,(\lambda_1, \dots,\lambda_n)$ and $\mathcal{Z}=(z_1 \dots z_n)$. Arguing as in the proof of Lemma \ref{AG} we have 
\begin{align}\label{equal2}
\chi_{_{\mathcal{U}\mathcal{S}}}(\lambda_j)= \chi'_{_\mathcal{U}}(\lambda_j) \,  z_j^* u v^T z_j\, .
\end{align} Let $a_j\neq0$ be the first component of the vector $z_j$. Then 
$$
 z_j^* u v^T z_j=z_j^* \mathcal{U}(\mathcal{I}-\mathcal{S})z_j=\lambda_j (1-\beta) \left|a_j\right|^2 \neq0\,.
$$
Thus the result follows from \eqref{equal2} and the first sentence of Lemma \ref{AG}.
\end{proof}

\begin{lemma}\label{Laux}
Let $\alpha_j \in \de$ ($j=0,1,\dots,n-1$) and $b_n \in \mathbb{S}^1$.  The following sentences hold:
\begin{enumerate}
\item [{\rm (i)}]  Let $\mathcal{S}$ be a diagonal matrix obtained from the $(n+1)$-by-$(n+1)$ identity matrix by replacing one of its diagonal entries with a number on $\mathbb{S}^1\setminus \{1\}$. Then $\mathcal{C}(\alpha_0,\dots,\alpha_{n-1},b_n)$ and $\mathcal{C}(\alpha_0,\dots,$ $\alpha_{n-1},b_n) \mathcal{S}$ have strictly interlacing eigenvalues on $\mathbb{S}^1$.
\item [{\rm (ii)}] Let $\mathcal{C}(\alpha_0,\dots,\alpha_{n-1},b_n)$ be partitioned as in \eqref{C}. Then, for each $0\leq m < n$, $\mathcal{C}_{22}$ has no eigenvalues on $\mathbb{S}^1$.
\end{enumerate}
\end{lemma}
\begin{proof}
(i) The result follows directly from Lemma \ref{ideafinal} and the fact that all the components of the eigenvectors of $\mathcal{C}(\alpha_0,\dots,\alpha_{n-1},b_n)$  are nonzero. %(cf. 

(ii) Assume that $m$ is even.  Note that $\mathcal{C}_{22}$ is the $(n-m)$-by-$(n-m)$ trailing principal submatrix of each of the matrices $\mathcal{C}(\alpha_m,\dots,\alpha_{n-1},b_n)$ and $\mathcal{C}(\alpha_m,\dots,\alpha_{n-1},b_n) \, \mathcal{S}$, where $\mathcal{S}:= \diag \, (\beta,\mathcal{I})$. Suppose the assertion (ii) is false. Since $ \mathcal{C}(\alpha_m,\dots,$ $\alpha_{n-1},b_{n})$ and  $\mathcal{C}(\alpha_m,\dots,\alpha_{n-1},b_{n})  \, \mathcal{S}$ are unitary matrices, these matrices share all the eigenvalues of $\mathcal{C}_{22}$ on $\mathbb{S}^1$, which contradicts sentence (i). If $m$ is odd, we argue in the same way noting that  $\mathcal{C}_{22}^T$ is the $(n-m)$-by-$(n-m)$ trailing principal submatrix of each of the matrices $\mathcal{C}(\alpha_m,\dots,\alpha_{n-1},b_n)$ and $\mathcal{S} \,\mathcal{C}(\alpha_m,\dots,\alpha_{n-1},b_n)$.
\end{proof}

\begin{lemma}\label{aux}
Let $\alpha_j \in \de$ ($j=0,1,\dots,n-1$) and $b_n \in \mathbb{S}^1$. Let $\mathcal{C}(\alpha_0,\dots,\alpha_{n-1},b_n)$ be partitioned as in \eqref{C}, where $0 \leq m<n$. Let $b_m \in \mathbb{S}^1$ and define  $b_m(\zeta)$ via \eqref{def} for each $\zeta \in \mathbb{S}^1$. Then $\mathcal{C}(\alpha_0,\dots,\alpha_{n-1},b_{n}) $ and $\mathcal{C}(\alpha_0,\dots,\alpha_{m-1},b_{m})$ have at most $\min \{m+1,n-m\}$ common eigenvalues, which consist of the set of different solutions $\zeta$ of the equation $b_m(\zeta)=b_m$ on $\sigma(\mathcal{C}(\alpha_0,\dots,\alpha_{m-1},b_m))$. \end{lemma}
\begin{proof} 
We begin by noting that
\begin{align}\label{equa}
\det \big( \zeta \mathcal{I}- \mathcal{C}_n \big)=\det \big( \zeta \mathcal{I}- \mathcal{C}(\alpha_0,\dots,\alpha_{m-1},b_m(\zeta)) \big) \det \big( \zeta \mathcal{I}-\mathcal{C}_{22}\big)
\end{align}
for each $\zeta \in \mathbb{S}^1$. Indeed, by sentence (ii) of Lemma \ref{Laux}, $\zeta \mathcal{I}-\mathcal{C}_{22}$ is nonsingular, hence \eqref{equa} follows from the equality (cf. \cite[Equation $3.41$]{B93})
 $$
 \mathcal{C}(\alpha_0,\dots,\alpha_{m-1},b_m(\zeta))=\mathcal{C}_{11}-\mathcal{C}_{12}(\mathcal{C}_{22}-\zeta \mathcal{I})^{-1}\mathcal{C}_{21}\, ,
 $$ 
after taking into account that the Schur complement of $\zeta \mathcal{I}-\mathcal{C}_{22}$ in $\zeta \mathcal{I}-\mathcal{C}(\alpha_0,\dots,\alpha_{n-1},$ $b_{n})$ is 
$\zeta \mathcal{I}-\big(\mathcal{C}_{11}-\mathcal{C}_{12}(\mathcal{C}_{22}-\zeta \mathcal{I})^{-1}\mathcal{C}_{21}\big)$. The result follows from \eqref{equa} and the fact that for $\nu, \zeta \in \mathbb{S}^1$, with $\nu \neq \zeta$, $\mathcal{C}(\alpha_0,\dots,\alpha_{m-1},\nu)$ and $\mathcal{C}(\alpha_0,\dots,\alpha_{m-1},\zeta)$ have no common eigenvalues (see e.g. \cite[Theorem $2.14.4$]{S11}; alternatively, apply sentence (i) of Lemma \ref{Laux}).  \end{proof}
%Let $\mathcal{S}$ be the diagonal matrix obtained from the identity matrix by replacing the $(m+1,m+1)$ entry with $\overline{\beta}$. Since $\mathcal{C}(\alpha_0,\dots,\alpha_{n-1},b_{n}) \, \mathcal{S}$ and $\mathcal{C}_m^\beta$ are unitarily similar (cf.  \cite[Theorem $5.1$]{S06}), the result follows from sentence (i) of Lemma \ref{Laux}.

\subsection{Proof of Theorem \ref{1}}\label{proof}
(i) Let $\mathcal{S}:=\diag \, (\mathcal{I}_{m}, \overline{\beta}, \mathcal{I}_{n-m})$, $\mathcal{D}:=\diag \, (\mathcal{I}_{m}, $ $\mathcal{J}_{n-m+1}^\beta)$, and $\mathcal{V}:=\diag \, (\mathcal{I}_{m+1}, $ $\mathcal{J}_{n-m}^\beta)$, where $\mathcal{J}_k^\beta:=\diag \, (\beta,1,\beta,1, \dots )$ is a $k$-by-$k$ diagonal matrix. Note that
\begin{align}\label{tab}
\begin{pmatrix}
\, \overline{\beta} & 0\\
0 & 1
\end{pmatrix} \Theta(\alpha) \begin{pmatrix}
1 & 0\\
0 & \beta
\end{pmatrix}=\Theta(\alpha \beta)\, .
\end{align}
Using \eqref{tab} it is easily seen that \footnote{A different proof is given in \cite[Theorem $5.1$]{S06}.}
\begin{align*}
\mathcal{D}^{*} \, \mathcal{C}(\alpha_0,\dots,\alpha_{n-1},b_{n}) \, \mathcal{D} \, \mathcal{S}=\big( \mathcal{D}^{*} \, \mathcal{L}\, \mathcal{V}\big) \, \big(\mathcal{V}^* \, \mathcal{M} \, \mathcal{D} \, \mathcal{S}\big)=\mathcal{C}_m^\beta\, ,
\end{align*}
when $m$ is even. Similarly, the transpose of \eqref{tab} leads to
\begin{align*}
\mathcal{S} \, \mathcal{D} \, \mathcal{C}(\alpha_0,\dots,\alpha_{n-1},b_{n}) \, \mathcal{D}^*=\big( \mathcal{S} \, \mathcal{D} \, \mathcal{L}\, \mathcal{V}^*\big) \, \big(\mathcal{V} \, \mathcal{M} \, \mathcal{D}^*\big)=\mathcal{C}_m^\beta\, ,
\end{align*}
when $m$ is odd. The result follows from sentence (i) of Lemma \ref{Laux}.

%The parentheses are used to indicate an appropriate order for operations.

(ii)  Define the block diagonal matrix $S:=\diag \, (\mathcal{I}_m, \mathcal{Z},\mathcal{I}_{n-m-1})$, where
$$
\mathcal{Z}= \Theta_{m}^*\begin{pmatrix}
\overline{b}_{m} &0 \\
0 & \gamma_{m}
\end{pmatrix}\,. 
$$
Hence
$$
\mathcal{C}(\alpha_0,\dots,\alpha_{m-1},b_{m}) \, \oplus \, \mathcal{N}=   \mathcal{C}(\alpha_0,\dots,\alpha_{n-1},b_{n}) \, \mathcal{S}\, ,
$$
when $m$ is odd, and
$$
\mathcal{C}(\alpha_0,\dots,\alpha_{m-1},b_{m}) \, \oplus \,  \mathcal{N}^T = \mathcal{S}^T \, \mathcal{C}(\alpha_0,\dots,\alpha_{n-1},b_{n})\, ,
$$
when $m$ is even. Note that $\mathcal{N}$ has simple eigenvalues (on $\mathbb{S}^1$) by sentence (i) of Lemma \ref{Laux}. The result follows from Lemma \ref{AG}, sentence (ii) of Lemma \ref{Laux}, and Lemma \ref{aux}.

\section*{Acknowledgment}
The authors thank the Bielefeld University Library for kindly sending them a hard copy of Birgit Bohnhorst's Ph.D. Thesis. KC is supported by the Portuguese Government through the Funda\c{c}\~ao para a Ci\^encia e a Tecnologia (FCT) under the grant SFRH/BPD/101139/2014. This work is partially supported by the Centre for Mathematics of the University of Coimbra -- UID/MAT/00324/2013, funded by the Portuguese Government through FCT/MCTES and co-funded by the European Regional Development Fund through the Partnership Agreement PT2020. JP is also partially supported by Direcci\'on General de Investigaci\'on Cient\'ifica y T\'ecnica, Ministerio de Econom\'ia y Competitividad of Spain under the project MTM2015--65888--C4--4--P.

% BibTeX users please use one of
\bibliographystyle{plain}      % basic style, author-year citations

\bibliography{bib}   % name your BibTeX data basearg\left(\frac{x(t)-i y(t)}{x(t)+i y(t)}\right)=

\begin{thebibliography}{10}

\bibitem{AGR88}
G.~Ammar, W.~Gragg, and L.~Reichel.
\newblock Constructing a unitary {H}essenberg matrix from spectral data.
\newblock In {\em Numerical {L}inear {A}lgebra, {D}igital {S}ignal {P}rocessing
  and {P}arallel {A}lgorithms ({L}euven, 1988)}, volume~70 of {\em NATO Adv.
  Sci. Inst. Ser. F Compt. Systems Sci.}, pages 385--395, Berlin, 1991.
  Springer.

\bibitem{AG92}
G.~S. Ammar and W.~B. Gragg.
\newblock Schur flow for orthogonal {H}essenberg matrices.
\newblock In {\em Hamiltonian and gradient flows, algorithms and control},
  volume~3 of {\em Fields Inst. Commun.}, pages 27--34, Providence, RI, 1994.
  Amer. Math. Soc.

\bibitem{AGR86}
G.~S. Ammar, W.~B. Gragg, and L.~Reichel.
\newblock On the eigenproblem for orthogonal matrices.
\newblock In {\em 25th IEEE Conference on Decision and Control}, pages
  1963--1966, Athens, Greece, 1986.

\bibitem{AGo88}
P.~Arbenz and G.~H. Golub.
\newblock On the spectral descomposition of {H}ermitian matrices modified by
  low rank perturbations with applications.
\newblock {\em SIAM J. Matrix Anal. Appl.}, 9:40--58, 1988.

\bibitem{A64}
F.~V. Atkinson.
\newblock {\em Discrete and continuous boundary problems}, volume~8 of {\em
  Mathematics in Science and Engineering}.
\newblock Academic Press, New York-London, 1964.

\bibitem{B96}
I.~Bar-{O}n.
\newblock Interlacing properties of tridiagonal symmetric matrices with
  applications to parallel computing.
\newblock {\em SIAM J. Matrix Anal. Appl.}, 17:548--562, 1996.

\bibitem{B93}
B.~Bohnhorst.
\newblock {\em Beitr\"age zur numerischen {B}ehandlung des unit\"aren
  {E}igenwertproblems}.
\newblock PhD thesis, Fakult\"at f\"ur {M}athematik, {U}niversit\"at Bielefeld,
  Bielefeld, Germany, 1993.

\bibitem{BBF00}
B.~Bohnhorst, A.~Bunse-Gerstner, and H.~Fa{\ss}bender.
\newblock On the perturbation theory for unitary eigenvalue problems.
\newblock {\em SIAM J. Matrix Anal. Appl.}, 21:809--824, 2000.

\bibitem{BE91}
A.~Bunse-Gerstner and L.~Elsner.
\newblock Schur parameter pencils for the solution of the unitary eigenproblem.
\newblock {\em Linear Algebra Appl.}, 154/156:741--778, 1991.

\bibitem{BH95}
A.~Bunse-Gerstner and C.~He.
\newblock On a {S}turm sequence of polynomials for unitary {H}essenberg
  matrices.
\newblock {\em SIAM J. Matrix Anal. Appl.}, 16:1043--1055, 1995.

\bibitem{CMV03}
M.~J. Cantero, L.~Moral, and L.~Vel\'azquez.
\newblock Five-diagonal matrices and zeros of orthogonal polynomials on the
  unit circle.
\newblock {\em Linear Algebra Appl.}, 362:29--56, 2003.

\bibitem{CCP16}
K.~Castillo, R.~Cruz-Barroso, and F.~Perdomo-P{\'\i}o.
\newblock On a spectral theorem in para-orthogonality theory.
\newblock {\em Pacific J. Math.}, 208:71--91, 2016.

\bibitem{DG88}
P.~Delsarte and Y.~Genin.
\newblock The tridiagonal approach to {S}zeg{\H{o}} orthogonal polynomials,
  {T}oeplitz linear systems, and related interpolation problems.
\newblock {\em SIAM J. Math. Anal.}, 19(3):718--735, 1988.

\bibitem{DG90}
P.~Delsarte and Y.~Genin.
\newblock On the role of orthogonal polynomials on the unit circle in digital
  signal processing applications.
\newblock In P.~Nevai, editor, {\em Orthogonal {P}olynomial: {T}heory and
  {P}ractice ({C}olumbos, {OH}, 1989)}, volume 294 of {\em NATO Adv. Sci. Inst.
  Ser. C Math. Phys. Sci.}, pages 115--133, Dordrecht, 1990. Kluwer Acad. Publ.

\bibitem{DG91a}
P.~Delsarte and Y.~Genin.
\newblock Tridiagonal approach to the algebraic environment of {T}oeplitz
  matrices. {I}. {B}asic results.
\newblock {\em SIAM J. Matrix Anal. Appl.}, 12(2):220--238, 1991.

\bibitem{DG91b}
P.~Delsarte and Y.~Genin.
\newblock Tridiagonal approach to the algebraic environment of {T}oeplitz
  matrices. {II}. {Z}eros and eigenvalue problems.
\newblock {\em SIAM J. Matrix Anal. Appl.}, 12(3):432--448, 1991.

\bibitem{G46}
Y.~L. Geronimus.
\newblock On the trigonometric moment problem.
\newblock {\em Ann. of Math.}, 47(2):742--761, 1946.

\bibitem{G44}
Ya.~L. Geronimus.
\newblock On polynomials orthogonal on the circle, on trigonometric
  moment-problem and on allied {C}arath\'eodory and {S}chur functions (in
  {R}ussian).
\newblock {\em Rec. Math. [Mat. Sbornik] N.S.}, 57(15):99--130, 1944.

\bibitem{G54}
Ya.~L. Geronimus.
\newblock Polynomials orthogonal on the unit circle and their applications.
\newblock In {\em Series and {A}pproximation}, volume~3 of {\em Series {O}ne},
  pages 1--78. Amer. Math. Soc., 1962.

\bibitem{G72}
G.~H. Golub.
\newblock {\em Some uses of the Lanczos algorithm in numerical linear algebra},
  pages 173--184.
\newblock Topics in numerical analysis (Proc. Roy. Irish Acad. Conf., Univ.
  Coll., Dublin, 1972). Academic Press, 1973.

\bibitem{GV83}
G.~H. Golub and C.~F. {Van Loan}.
\newblock {\em Matrix Computation}.
\newblock Johns Hopkins Studies in the Mathematical Sciences. The Johns Hopkins
  University Press, Baltimore, MD, fourth edition, 2013.

\bibitem{G82o}
W.~B. Gragg.
\newblock {\em Positive definite {T}oeplitz matrices, the {H}essenberg process
  for isometric operators, and the {G}auss quadrature on the unit circle (in
  Russian)}, pages 16--32.
\newblock Numerical {M}ethods in {L}inear {A}lgebra. Moskov. Gos. Univ.,
  Moscow, 1982.

\bibitem{G86}
W.~B. Gragg.
\newblock The {QR} algorithm for unitary {H}essenberg matrices.
\newblock {\em J. Comp. Appl. Math.}, 16:1--8, 1986.

\bibitem{G82}
W.~B. Gragg.
\newblock Positive definite {T}oeplitz matrices, the {H}essenberg process for
  isometric operators, and the {G}auss quadrature on the unit circle.
\newblock {\em J. Comp. Appl. Math.}, 46:183--198, 1993.

\bibitem{GR90}
W.~B. Gragg and L.~Reichel.
\newblock A divide and conquer method for unitary and orthogonal eigenproblems.
\newblock {\em Numer. Math.}, 57:695--718, 1990.

\bibitem{HP92}
R.~O. Hill{ {J}r.} and B.~N. Parlett.
\newblock Refined interlacing properties.
\newblock {\em SIAM J. Matrix Anal. Appl.}, 13:239--247, 1992.

\bibitem{HJ}
R.~A. Horn and C.~R. Johnson.
\newblock {\em Matrix Analysis}.
\newblock Cambridge University Press, New York, second edition, 2013.

\bibitem{JNT89}
W.~B. Jones, O.~Nj{\aa}stad, and W.~J. Thron.
\newblock Moment theory, orthogonal polynomials, quadrature, and continued
  fractions associated with the unit circle.
\newblock {\em Bull. London Math. Soc.}, 21:113--152, 1989.

\bibitem{Ka66}
W.~Kahan.
\newblock Accurate eigenvalues of a symmetric tridiagonal matrix.
\newblock Tech. report CS41, Stanford University, Stanford, CA, 1966.

\bibitem{KN07}
R.~Killip and I.~Nenciu.
\newblock {CMV}: The unitary analogue of {J}acobi matrices.
\newblock {\em Comm. Pure Appl. Math.}, LX:1148--1188, 2007.

\bibitem{K85}
H.~Kimura.
\newblock Generalized {S}chwarz form and lattice-ladder realizations of digital
  filters.
\newblock {\em IEEE Trans. Circuits Systems}, 32:1130--1139, 1985.

\bibitem{MMRR14}
C.~Mehl, V.~Mehrmann, A.~C.~M. Ran, and L.~Rodman.
\newblock Eigenvalue perturbation theory of symplectic, orthogonal, and unitary
  matrices under generic structured rank one perturbations.
\newblock {\em BIT Numer. Math.}, 54:219--255, 2014.

\bibitem{S79}
D.~S. Scott.
\newblock How to make the {L}anczos algorithm converge slowly.
\newblock {\em Math. Comp.}, 33:239--247, 1979.

\bibitem{S10}
D.~Serre.
\newblock {\em Matrices: {T}heory and {A}pplications}.
\newblock Graduate Texts in Mathematics. Springer-Verlag, New York, second
  edition, 2010.

\bibitem{S05I}
B.~Simon.
\newblock {\em Orthogonal polynomials on the unit circle. {P}art {I}.
  {C}lassical {T}heory}, volume~54 of {\em Amer. Math. Soc. Coll. Publ.}
\newblock Amer. Math. Soc., Providence, RI, 2005.

\bibitem{S05II}
B.~Simon.
\newblock {\em Orthogonal polynomials on the unit circle. {P}art {II}.
  {S}pectral {T}heory}, volume~54 of {\em Amer. Math. Soc. Coll. Publ.}
\newblock Amer. Math. Soc., Providence, RI, 2005.

\bibitem{S06}
B.~Simon.
\newblock Aizenman's theorem for orthogonal polynomials on the unit circle.
\newblock {\em Constr. Approx.}, 23:229--240, 2006.

\bibitem{S07b}
B.~Simon.
\newblock {CMV} matrices: {F}ive years after.
\newblock {\em J. Comp. Appl. Math.}, 208:120--154, 2007.

\bibitem{S07}
B.~Simon.
\newblock Rank one perturbations and the zeros of paraorthogonal polynomials on
  the unit circle.
\newblock {\em J. Math. Anal. Appl.}, 329:376--382, 2007.

\bibitem{S11}
B.~Simon.
\newblock {\em Szeg\H{o}'s theorem and its descendants: {S}pectral theory for
  $L^2$ perturbations of orthogonal polynomials}.
\newblock M. B. Porter Lectures. Princeton University Press, Princeton, 2011.

\bibitem{TM68}
R.~C. Thompson and P.~Mc{E}nteggert.
\newblock Principal submatrices. {II}: {T}he upper and lower quadratic
  inequalities.
\newblock {\em Linear Algebra Appl.}, 1:211--243, 1968.

\bibitem{T74o}
P.~Tur\'an.
\newblock Some open problems in approximation theory (in {H}ungarian).
\newblock {\em Mat. Lapok}, 25:21--75, 1974.

\bibitem{T80}
P.~Tur\'an.
\newblock On some open problems of approximation theory. {P}. {T}ur\'an
  memorial volume. {T}ranslated from the {H}ungarian by {P}. {S}z\"usz.
\newblock {\em J. Approx. Theory}, 29(1):23--85, 1980.

\bibitem{W93}
D.~S. Watkins.
\newblock Some perspectives on the eigenvalue problem.
\newblock {\em SIAM Rev.}, 35(3):430--471, 1993.

\bibitem{W88}
J.~H. Wilkinson.
\newblock {\em The algebraic eigenvalue problem}.
\newblock Monographs on Numerical Analysis. Oxford Science Publications. The
  Clarendon Press, Oxford University Press, New York, 1988.

\end{thebibliography}

\end{document}